\documentclass{amsart}
\usepackage{amssymb}
\usepackage{amscd}
\usepackage{multicol}
\usepackage{tikz}
\usepackage{graphics}
\usepackage{shuffle}
\usepackage[OT2,T1]{fontenc}
\usepackage[all]{xy}
\usepackage{extarrows}
\usepackage{comment}
\DeclareSymbolFont{cyrletters}{OT2}{wncyr}{m}{n}
\DeclareMathSymbol{\Sha}{\mathalpha}{cyrletters}{"58}

\title[Equivalence between desingularized and renormalized values]
{An equivalence between desingularized and renormalized values of multiple zeta functions at negative integers\\
}
\author{Nao Komiyama}

\address{Graduate School of Mathematics, Nagoya University, 
Furo-cho, Chikusa-ku, Nagoya 464-8602 Japan }
\email{m15027u@math.nagoya-u.ac.jp}
\thanks{{\tt Submitted}}
\date{April 11, 2017}

\newtheorem{thm}{Theorem}[section]
\newtheorem{lem}[thm]{Lemma}
\newtheorem{cor}[thm]{Corollary}
\newtheorem{prop}[thm]{Proposition}  

\theoremstyle{remark}
        

\numberwithin{equation}{section}
\theoremstyle{definition}
\newtheorem{definition}[thm]{Definition}
\newtheorem{remark}[thm]{Remark}

\newtheorem{example}[thm]{Examples}

\newtheorem{prob}[thm]{Problem}    

\newcommand{\twoheadlongrightarrow}{\relbar\joinrel\twoheadrightarrow}

\begin{document}
\bibliographystyle{amsalpha+}
\maketitle

\begin{abstract}
	It is known that the special values of multiple zeta functions at non-positive arguments are indeterminate in most cases due to the occurrences of infinitely many singularities. In order to give a suitable rigorous meaning of the special values there, Furusho, Komori, Matsumoto and Tsumura introduced the desingularized values by the desingularization method to resolve all singularities. While, Ebrahimi-Fard, Manchon and Singer introduced the renormalized values to keep the ``shuffle'' relation by the renormalization procedure \`a la Connes and Kreimer. In this paper, we reveal an equivalence, that is, an explicit interrelationship between these two values. As a corollary, we also obtain an explicit formula to describe renormalized values in terms of Bernoulli numbers.
\end{abstract}

\tableofcontents
\setcounter{section}{-1}
\section{Introduction}
In 1776, Euler (\cite{Euler}) considered a certain power series, the so-called double zeta values, and showed several relations among them. More than 200 years later than Euler, the {\it multiple zeta value} (MZV for short) which is more general series
$$\zeta(k_1, \dots, k_n) := \sum_{0<m_1<\cdots<m_n}\frac{1}{m_1^{k_1}\cdots m_n^{k_n}}$$
converging for $k_1,\cdots,k_n\in\mathbb{N}$ and $k_n>1$, appeared in \cite{Eca} written by Ecalle again, in 1981. In 1990s, these values also came to be focused by Hoffman (\cite{Hof}) and Zagier (\cite{Zag}). The MZV admits an iterated integral expression, which enables us to regard it as a period of a certain motive. (\cite{DG}, \cite{Go} and \cite{Te}). MZVs appear in calculations of the Kontsevich invariant in knot theory (\cite{CDM} and \cite{LM}). MZVs are also related to mathematical physics in \cite{BK1} and \cite{BK2}. They are explained in \cite{Zhao2}.

MZVs are regarded as special values at positive integer points of the {\it multiple zeta-function} (MZF for short), the series  
\begin{equation}\label{eqn:0.1}
	\zeta(s_1, \dots, s_n) := \sum_{0<m_1<\cdots<m_n}\frac{1}{m_1^{s_1}\cdots m_n^{s_n}}
\end{equation}
which converges absolutely in the region
\begin{equation*}
	\{(s_1,\cdots,s_n)\in\mathbb{C}^n\ |\ \frak{R}(s_{n-k+1}+\cdots+s_n)>k\ (1\leq k\leq n)\}.
\end{equation*}
 In the early 2000s, Zhao (\cite{Zhao}) and Akiyama, Egami and Tanigawa (\cite{AET}) independently showed that MZF can be meromorphically continued to $\mathbb{C}^n$. Especially, in \cite{AET}, the set of all singularities of the function $\zeta(s_1,\cdots,s_n)$ is determined as
	\begin{align}\label{eqn:0.2}
		&s_n=1,\nonumber\\
		&s_{n-1}+s_n=2,1,0,-2,-4,\cdots,\\
		&s_{n-k+1}+\cdots+s_n=k-r\quad (3\leq k\leq n,\ r\in\mathbb{N}_0).\nonumber
	\end{align}
Because almost all of integer points with non-positive arguments are located in the above singularities, the special values of MZF there are indeterminate in all cases except for $\zeta(-k)$ at $k\in\mathbb{N}_0$, and $\zeta(-k_1,-k_2)$ at $k_1,k_2\in\mathbb{N}_0$ with $k_1+k_2$ odd. Actually, giving a nice definition of ``$\zeta(-k_1,\dots,-k_n)$'' for $k_1,\dots,k_n\in\mathbb{N}_0$ is one of our most fundamental problems.

In order to resolve all infinitely many singularities of MZF, the desingularization method was introduced  by Furusho, Komori, Matsumoto and Tsumura in \cite{FKMT}. By applying this method to $\zeta(s_1,\dots,s_n)$, they constructed the {\it desingularized MZF}\footnote{It is denoted by $\zeta_n^{\rm des}((s_j);(1))$ in \cite{FKMT}.} $\zeta_{\scalebox{0.5}{\rm FKMT}}(s_1,\dots,s_n)$ which is entire on the whole space $\mathbb{C}^n$ and they also showed its basic properties. The {\it desingularized value}
\begin{equation}\label{eqn:0.6}
\zeta_{\scalebox{0.5}{\rm FKMT}}(-k_1,\dots,-k_n)\in\mathbb{C}
\end{equation}
is given as the special value of $\zeta_{\scalebox{0.5}{\rm FKMT}}(s_1,\dots,s_n)$ at $(s_1,\dots,s_n)=(-k_1,\dots,-k_n)$ for $k_1,\cdots,k_n\in\mathbb{N}_0$ (see Definition \ref{def:1.2.1}). In \cite{FKMT}, its generating function given by
\begin{equation}\label{eqn:0.4}
	Z_{\scalebox{0.5}{\rm FKMT}}(t_1,\dots,t_n) := \sum_{k_1,\dots,k_n=0}^{\infty}\frac{(-t_1)^{k_1}\cdots(-t_n)^{k_n}}{k_1!\cdots k_n!}\zeta_{\scalebox{0.5}{\rm FKMT}}(-k_1,\dots,-k_n)
\end{equation}
 in $\mathbb{C}[[t_1,\cdots,t_n]]$ was calculated and the desingularized values were described in terms of the Bernoulli numbers. (See Proposition \ref{prop:1.1.1}.)

In contrast, Connes and Kreimer (\cite{CK}) started a Hopf algebraic approach to the renormalization procedure in the perturbative quantum field theory. A fundamental tool in their work is the {\it algebraic Birkhoff decomposition} (Theorem \ref{thm:2.2.1}). By applying this decomposition to a certain Hopf algebra parameterizing regularized MZVs, Guo and Zhang (\cite{GZ}) gave the {\it renormalized values} which satisfy  the harmonic relations. Later, Manchon and Paycha (\cite{MP}) and Ebrahimi-Fard, Manchon and Singer (\cite{EMS2}) introduced the different renormalized values which obey  harmonic(-like) relations by using different Hopf algebras. Meanwhile, Ebrahimi-Fard, Manchon and Singer (\cite{EMS1}) also introduced another type of the renormalized values (cf. Definition \ref{def:2.3.1}) satisfying the ``shuffle relations'' (see Proposition \ref{prop:2.3.1} for precise), which in this paper we denote as
\begin{equation}\label{eqn:0.7}
\zeta_{\scalebox{0.5}{\rm EMS}}(-k_1,\dots,-k_n)\in\mathbb{C}
\end{equation}
for $k_1,\dots,k_n\in \mathbb{N}_0$, and which we consider with its generating function given by
\begin{equation}\label{eqn:0.5}
	Z_{\scalebox{0.5}{\rm EMS}}(t_1,\dots,t_n) := \sum_{k_1,\dots,k_n=0}^{\infty}\frac{(-t_1)^{k_1}\cdots(-t_n)^{k_n}}{k_1!\cdots k_n!}\zeta_{\scalebox{0.5}{\rm EMS}}(-k_1,\dots,-k_n)
\end{equation}
 in $\mathbb{C}[[t_1,\cdots,t_n]]$.

Our main theorem in this paper is an equivalence between the desingularized values (\ref{eqn:0.6}) and the renormalized values (\ref{eqn:0.7}):\\

\noindent
\smallskip
{\bf Theorem \ref{thm:3.2.1}.} {\it For $n\in\mathbb{N}$, we have}
\begin{equation*}
	Z_{\scalebox{0.5}{\rm EMS}}(t_1,\dots,t_n) = \prod_{i=1}^{n}\frac{1-e^{-t_i-\cdots-t_n}}{t_i+\cdots+t_n}\cdot Z_{\scalebox{0.5}{\rm FKMT}}(-t_1,\dots,-t_n).
\end{equation*}
\smallskip

As a consequence of this theorem, the renormalized values can be given as linear combinations of
the desingularized values and vice versa (cf. Examples \ref{ex:3.1} and \ref{ex:3.2}). By combining the above equivalence with the explicit formula (cf. Proposition \ref{prop:1.1.1}) of the desingularized values shown in \cite{FKMT}, we obtain the following explicit formula of the renormalized values.\\

\noindent
\smallskip
{\bf Corollary \ref{cor:3.2.1}.} 
For $k_1,\cdots,k_n\in\mathbb{N}_0$, we have
\begin{equation}\label{eqn:0.3}
	\zeta_{\scalebox{0.5}{\rm EMS}}(-k_1,\dots,-k_n)=(-1)^{k_1+\cdots+k_n}\sum_{\substack{\nu_{1i}+\cdots+\nu_{ii}=k_i\\1\leq i\leq n}}\prod_{i=1}^n\frac{k_i!}{\prod_{j=i}^n\nu_{ij}!}\frac{B_{\nu_{ii}+\cdots+\nu_{in}+1}}{\nu_{ii}+\cdots+\nu_{in}+1}.
\end{equation}
Here $B_n$ is the Bernoulli number in (\ref{eqn:1.1.4}).\\
\smallskip

\noindent

The plan of our paper goes as follows. In section \ref{sec:1}, we recall the desingularization method, desingularized MZF and the desingularized values introduced by Furusho, Komori, Matsumoto and Tsumura in \cite{FKMT}. In section \ref{sec:2}, we review an algebraic framework on Hopf algebra in \cite{EMS1}, and we prove an explicit formula of the reduced coproduct $\tilde{\Delta}_0$ (Proposition \ref{cor:2.1.1}) which is required to prove the recurrence formula of renormalized values in \cite{EMS1} in section \ref{sec:3}. We also review the algebraic Birkhoff decomposition and renormalized values in \cite{EMS1}. In section \ref{sec:3}, by showing a recurrence formula (Proposition \ref{thm:3.1.1}) we prove the above main results, that is, an equivalence between desingularized values and renormalized values (Theorem \ref{thm:3.2.1}) and an explicit formula of renormalized values (Corollary \ref{cor:3.2.1}).

\section{Desingularizations}\label{sec:1}
In this section, we review the desingularized values introduced by Furusho, Komori, Matsumoto and Tsumura in \cite{FKMT}. In \S1.1 we recall the desingularization method and desingularized MZF, and explain  some remarkable properties of this function. In \S1.2, we review the desingularized values and their generating function.
\subsection{The desingularization method and desingularized MZFs}
In this subsection, we review the desingularization method, the desingularized MZF. We also recall the basic properties of the desingularized MZF. 

The desingularization method is a method to resolve all singularities of MZF. We recall the generating function\footnote{It is denoted by $\tilde{\mathfrak{H}}_n\left((t_j);(1);c\right)$ in \cite{FKMT}.} $\tilde{\mathfrak{H}}_n\left(t_1,\dots,t_n;c\right) \in \mathbb{C}[[t_1,\dots,t_n]]$ which is defined by in \cite{FKMT} Definition 1.9
\begin{align*}
	\tilde{\mathfrak{H}}_n\left(t_1,\dots,t_n;c\right)&:=\prod_{j=1}^n\left(\frac{1}{\exp{\left(\sum_{k=j}^n t_k\right)}-1}-\frac{c}{\exp{\left(c\sum_{k=j}^n t_k\right)}-1}\right)\\
	&=\prod_{j=1}^n\left(\sum_{m=1}^{\infty}(1-c^m)B_m\frac{\left(\sum_{k=j}^n t_k\right)^{m-1}}{m!}\right)
\end{align*}
for $c\in\mathbb{R}$. Here $B_m\ (m\geq0)$ is the Bernoulli number which is defined by
\begin{equation}\label{eqn:1.1.4}
\displaystyle\frac{x}{e^x-1}:=\sum_{m\geq0}\frac{B_m}{m!}x^m.
\end{equation}
We note that $B_0=1$, $B_1=-\frac{1}{2}$, $B_2=\frac{1}{6}$.
\begin{definition}[\cite{FKMT} Definition 3.1]
	For non-integral complex numbers $s_1,\dots,s_n$, the {\it desingularized MZF} $\zeta_{\scalebox{0.5}{\rm FKMT}}(s_1,\dots,s_n)$ is defined by
	\begin{align}
		\label{eqn:1.1.2}&\zeta_{\scalebox{0.5}{\rm FKMT}}(s_1,\dots,s_n) \\
		&:=\lim_{\substack{c\rightarrow1\\c\in\mathbb{R}\setminus\{1\}}}\frac{1}{(1-c)^n}\prod_{k=1}^n\frac{1}{(e^{2\pi is_k}-1)\Gamma(s_k)}\int_{\mathcal{C}^n}\tilde{\mathfrak{H}}_n\left(t_1,\dots,t_n;c\right)\prod_{k=1}^n t_k^{s_k-1}d t_k. \nonumber
	\end{align}
	Here $\mathcal{C}$ is the path consisting of the positive real axis (top side), a circle around the origin of radius $\varepsilon$ (sufficiently small), and the positive real axis (bottom side).
\end{definition}
One of the remarkable properties of the desingularized MZF is that it is an entire function, i.e., the equation (\ref{eqn:1.1.2}) is well-defined as an analytic function by the following proposition.
\begin{prop}[\cite{FKMT} Theorem 3.4]
	The equation $\zeta_{\scalebox{0.5}{\rm FKMT}} (s_1,\dots,s_n)$ can be analytically continued to $\mathbb{C}^n$ as an entire function in $(s_1,\dots,s_n)\in \mathbb{C}^n$ by the following integral expression:
	\begin{align*}
		\label{eqn:1.1.2}&\zeta_{\scalebox{0.5}{\rm FKMT}}(s_1,\dots,s_n) \\
		&=\prod_{k=1}^n\frac{1}{(e^{2\pi is_k}-1)\Gamma(s_k)}\\
		&\times\int_{\mathcal{C}^n}\prod_{j=1}^n\lim_{\substack{c\rightarrow1\\c\in\mathbb{R}\setminus\{1\}}}\frac{1}{1-c}\left(\frac{1}{\exp{\left(\sum_{k=j}^n t_k\right)}-1}-\frac{c}{\exp{\left(c\sum_{k=j}^n t_k\right)}-1}\right)\prod_{k=1}^n t_k^{s_k-1}d t_k.
	\end{align*}
\end{prop}
We explain another remarkable properties of the desingularized MZF. For indeterminates $u_j$ and $v_j\ (1\leq j\leq n)$, we set
\begin{equation*}
	\mathcal{G}((u_j),(v_j)):=\prod_{j=1}^n\left(1-(u_jv_j+\cdots+u_n v_n)(v_j^{-1}-v_{j-1}^{-1})\right)
\end{equation*}
with the convention $v_0^{-1}:=0$, and we define the set of integers $\{a_{\mbox{\boldmath {\footnotesize$l$}},\mbox{\boldmath {\footnotesize$m$}}}\}$ by
\begin{equation*}
	\mathcal{G}((u_j),(v_j))=\sum_{\substack{\mbox{\boldmath {\footnotesize$l$}}=(l_j)\in\mathbb{N}_0^n\\ \mbox{\boldmath {\footnotesize$m$}}=(m_j)\in\mathbb{Z}^n \\ \sum_{j=1}^n m_j=0}}a_{\mbox{\boldmath {\footnotesize$l$}},\mbox{\boldmath {\footnotesize$m$}}}\prod_{j=1}^nu_j^{l_j}v_j^{m_j}.
\end{equation*}
Another remarkable properties of the desingularized MZF is that the function is given by a finite `linear' combination of MZFs.
\begin{prop}[\cite{FKMT} Theorem 3.8]
	For $s_1,\dots,s_n \in \mathbb{C}$, we have the following equality between meromorphic functions of the complex variables $(s_1,\ldots,s_n)$.
	\begin{equation*}
		\zeta_{\scalebox{0.5}{\rm FKMT}}(s_1,\dots,s_n)=\sum_{\substack{\mbox{\boldmath {\footnotesize$l$}}=(l_j)\in\mathbb{N}_0^n\\ \mbox{\boldmath {\footnotesize$m$}}=(m_j)\in\mathbb{Z}^n \\ \sum_{j=1}^n m_j=0}}a_{\mbox{\boldmath {\footnotesize$l$}},\mbox{\boldmath {\footnotesize$m$}}}\left(\prod_{j=1}^n(s_j)_{l_j}\right)\zeta(s_1+m_1,\dots,s_n+m_n).
	\end{equation*}
	Here, $(s)_{k}$ is the {\it Pochhammer symbol}, that is, for $k\in\mathbb{N}$ and $s\in\mathbb{C}$ $(s)_{0}:=1$ and $(s)_k:=s(s+1)\cdots(s+k-1)$.
\end{prop}
\subsection{Desingularized values}
We review the desingularized values and its explicit formula (Proposition \ref{prop:1.1.1}), and then we give a recurrence formula of the desingularized values (Corollary \ref{cor:1.1.1}).

The desingularized value is given as the special value at the integer points with non-positive arguments of an entire function:

\begin{definition}\label{def:1.2.1}
	For $k_1,\dots,k_n \in \mathbb{N}_0$, the {\it desingularized value} $\zeta_{\scalebox{0.5}{\rm FKMT}}(-k_1,\dots,-k_n)\in\mathbb{C}$ is defined to be the special value of desingularized MZF $\zeta_{\scalebox{0.5}{\rm FKMT}}(s_1,\dots,s_n)$ at $(s_1,\dots,s_n)=(-k_1,\dots,-k_n)$.
\end{definition}
The generating function $Z_{\scalebox{0.5}{\rm FKMT}}(t_1,\dots,t_n)$ of $\zeta_{\scalebox{0.5}{\rm FKMT}}(-k_1,\dots,-k_n)$ in the equation (\ref{eqn:0.4}) is explicitly calculated as follows.
\begin{prop}[\cite{FKMT} Theorem 3.7]\label{prop:1.1.1}
	We have
	\begin{equation*}
		Z_{\scalebox{0.5}{\rm FKMT}}(t_1,\dots,t_n) = \prod_{i=1}^n\frac{(1-t_i-\cdots-t_n)e^{t_i+\cdots+t_n}-1}{(e^{t_i+\cdots+t_n}-1)^2}.
	\end{equation*}
	In terms of $\zeta_{\scalebox{0.5}{\rm FKMT}}(-k_1,\dots,-k_n)$ for $k_1,\dots,k_n\in\mathbb{N}_0$, the above equation is reformulated to
	\begin{equation}
		\zeta_{\scalebox{0.5}{\rm FKMT}}(-k_1,\dots,-k_n)=(-1)^{k_1+\cdots+k_n}\sum_{\substack{\nu_{1i}+\cdots+\nu_{ii}=k_i\\1\leq i\leq n}}\prod_{i=1}^n\frac{k_i!}{\prod_{j=i}^n\nu_{ij}!}B_{\nu_{ii}+\cdots+\nu_{in}+1}.
	\end{equation}
\end{prop}
By the above proposition we have the following recurrence formula:
\begin{cor}\label{cor:1.1.1}
	\begin{equation}\label{eqn:1.2.1}
		Z_{\scalebox{0.5}{\rm FKMT}}(t_1,\dots,t_n) = Z_{\scalebox{0.5}{\rm FKMT}}(t_2,\dots,t_n)\cdot Z_{\scalebox{0.5}{\rm FKMT}}(t_1+\cdots+t_n) \quad(n \in \mathbb{N}).
	\end{equation}
	In terms of $\zeta_{\scalebox{0.5}{\rm FKMT}}(-k_1,\dots,-k_n)$, the equation {\rm (\ref{eqn:1.2.1})} is reformulated to
	\begin{equation}\label{eqn:1.2.2}
		\zeta_{\scalebox{0.5}{\rm FKMT}}(-k_1,\dots,-k_n) = \sum_{\substack{i_2 + j_2=k_2\\\scalebox{0.5}{\rotatebox{90}{$\cdots$}}\\i_n + j_n=k_n}}\prod_{a=2}^n\binom{k_a}{i_a}\zeta_{\scalebox{0.5}{\rm FKMT}}(-i_2,\dots,-i_n)\zeta_{\scalebox{0.5}{\rm FKMT}}(-k_1-j_2-\dots-j_n)
	\end{equation}
	for $k_1,\dots,k_n \in \mathbb{N}_0$. Here we use $\binom{k_a}{i_a}:=\frac{k_a!}{i_a!(k_a-i_a)!}$.
\end{cor}
In \S \ref{sec:3}, we will show that the same formula as (\ref{eqn:1.2.2}) holds for the renormalized value $\zeta_{\scalebox{0.5}{\rm EMS}}(-k_1,\dots,-k_n)$ in the equation (\ref{eqn:3.2.1}).
\section{Renormalizations}\label{sec:2}
In this section, we recall the renormalization procedure to define renormalized values which is introduced by Ebrahimi-Fard, Manchon and Singer. In \S 2.1, we start by recalling their framework of a Hopf algebra generated by words and in \S 2.2 we show an explicit formula in Proposition \ref{cor:2.1.1} to calculate the reduced coproduct $\tilde{\Delta}_0$. This proposition is essential to show the recurrence formula of $\zeta_{\scalebox{0.5}{\rm EMS}}(-k_1,\dots, -k_n)$ in \S3. In \S2.3 we explain the algebraic Birkhoff decomposition \`a la Connes and Kreimer which is required to define renormalized values.
\subsection{Algebraic frameworks}
We follow the conventions of \cite{EMS1}. Let $X_0:=\{j,d,y\}$ be the set of three elements $j$, $d$ and $y$. Let $W_0$ be the associative monoid, with the empty word {\bf 1} as a unit, generated by $X_0$ with the rule $j d=d j={\bf 1}$. Any element $w\in W_0$ can be uniquely represented by
$$w=j^{k_1}y\cdots j^{k_n}$$
for $k_1,\cdots,k_n\in\mathbb{Z}$. An element of $W_0$ is called a {\it word}. Put $Y_0:=W_0y\cup\{{\bf 1}\}$ and we call an element of $Y_0$ {\it admissible}. We denote the $\mathbb{Q}$-linear space $\mathcal{A}_0$ generated by $W_0$ by $\mathcal{A}_0:=\langle W_0\rangle_{\mathbb{Q}}$. The linear space $\mathcal{A}_0$ is naturally equipped with a structure of a non-commutative algebra. We equip this $\mathcal{A}_0$ with a new product $\shuffle_0$ :$\mathcal{A}_0\otimes\mathcal{A}_0\rightarrow\mathcal{A}_0$ which is a $\mathbb{Q}$-linear map recursively defined by
\begin{align*}
	{\bf 1}\shuffle_0w&:=w\shuffle_0{\bf 1} := w\quad (w\in W_0),\\
	y u\shuffle_0v&:= u\shuffle_0y v:= y(u\ \shuffle_0\ v)\quad (u,v \in W_0),\\
	j u\shuffle_0j v&:= j(u\shuffle_0j v) +j(j u\shuffle_0v) \quad(u,v \in W_0),\\
	d u\shuffle_0d v&:= d(u\shuffle_0d v) -u\shuffle_0d^2v \quad(u,v \in W_0).
\end{align*}
Then $(\mathcal{A}_0, \shuffle_0)$ forms a unitary, nonassociative, noncommutative $\mathbb{Q}$-algebra. We define 
\begin{equation*}
	\mathcal{T}:= \langle \{j^{k_1}y\cdots j^{k_{n-1}}y j^{k_n}\in W_0\ | \ k_n\neq0, n\in\mathbb{N}\}\rangle_{\mathbb{Q}},
\end{equation*}
that is, to be the linear subspace of $\mathcal{A}_0$ linearly generated by words ending in $d$ or $j$ and 
\begin{equation*}
	\mathcal{L}:= \langle j^k\{d(u \ \shuffle_0 \ v) -d u \ \shuffle_0 \ v - u \ \shuffle_0 \ d v\} \ |\ k \in \mathbb{Z},\ u,v \in W_0y\ \rangle_{(\mathcal{A}_0,\shuffle_0)},
\end{equation*}
that is, to be the two-sided ideal of $(\mathcal{A}_0,\shuffle_0)$ algebraically generated by the above elements. The subspace $\mathcal{T}$ forms a two-sided ideal of $\mathcal{A}_0$ by \cite{EMS1} Lemma 3.4. We define the quotient algebra
\begin{equation*}
	\mathcal{B}_0':=\mathcal{A}_0/(\mathcal{T}+\mathcal{L}).
\end{equation*}
We consider the map 
\begin{equation}\label{eqn:2.1.6}
\zeta^{\shuffle}_t:\mathcal{B}_0'\rightarrow \mathbb{Q}[[t]]
\end{equation}
by $\zeta^{\shuffle}_t({\bf 1}):=1$ and for $k_1,\dots,k_n\in\mathbb{Z}$,
	$$\zeta^{\shuffle}_t(j^{k_n}y\cdots j^{k_1}y):={\rm Li}_{k_1,\cdots,k_n}(t).$$
Here ${\rm Li}_{k_1,\cdots,k_n}(t)$ is the {\it multiple polylogarithm} defined by
	$${\rm Li}_{k_1,\cdots,k_n}(t):=\displaystyle\sum_{0<m_1<\cdots<m_n}\frac{t^{m_n}}{m_1^{k_1}\cdots m_n^{k_n}}.$$
\begin{lem}\label{lem:2.1.1}
	The map $\zeta^{\shuffle}_t$ is well-defined and forms an algebra homomorphism.
\end{lem}
The first half of the claim of Lemma \ref{lem:2.1.1} is proved in the same way to proof of \cite{EMS1} Proposition 3.5 and the latter half of the claim of Lemma \ref{lem:2.1.1} is proved in \cite{EMS1} Lemma 3.6.

\begin{remark}
	The restriction of the shuffle product $\shuffle_0$ to admissible words at positive arguments corresponds the usual shuffle product $\shuffle$ as is proved in \cite{EMS1} Lemma 3.7. Let $\mathcal{C}:=\mathbb{Q}\oplus j\mathbb{Q}\langle j,y\rangle y$ and $\mathcal{D}:=\mathbb{Q}\oplus x_0\mathbb{Q}\langle x_0,x_1\rangle x_1$. Then two algebras $(\mathcal{C},\shuffle_0)$ and $(\mathcal{D},\shuffle)$ become isomorphic under the linear map $\Phi:(\mathcal{D},\shuffle)\rightarrow(\mathcal{C},\shuffle_0)$ by $\Phi({\bf 1}):={\bf 1}$ and for $k_1,\dots,k_n\in\mathbb{N}$ with $k_1>1$,
		$$\Phi(x_0^{k_1-1}x_1\cdots x_0^{k_n-1}x_1):=j^{k_1-1}y\cdots j^{k_n-1}y.$$
\end{remark}
 Let $L := \{d,y\}$ be the set of two elements $d$ and $y$. Let $L^*$ be the free monoid of $L$ with empty word {\bf 1} as a unit. This $L^*$ forms a submonoid of $W_0$. Put $Y := L^*y \cup \{{\bf 1}\}\subset Y_0$. So all elements of $Y$ are admissible. The {\it weight} ${\rm wt}(w)$ of a word $w \in L^*$ means the number of letters appearing in $w$ and the {\it depth} ${\rm dp}(w)$ of a word $w \in L^*$ is given by the number of $y$ appearing in $w$. We denote the free unitary, associative, noncommutative $\mathbb{Q}$-algebra of $L$ by $\mathbb{Q}\langle L\rangle$. Then $(\mathbb{Q}\langle L\rangle, \shuffle_0)$ forms a unitary, nonassociative, noncommutative $\mathbb{Q}$-subalgebra of $\mathcal{A}_0$. The algebra $\mathbb{Q}\langle L\rangle$ also forms a counital, cocommutative coalgebra. (See \cite{EMS1} \S3.3.5.) We define 
\begin{equation*}
	\mathcal{T}_-:= \langle \{wd\ | \ w \in L^*\}\rangle_{\mathbb{Q}}\ \bigl(=\mathcal{T}\cap \mathbb{Q}\langle L\rangle\bigr),
\end{equation*}
that is, to be the linear subspace of $\mathbb{Q}\langle L\rangle$ linearly generated by words ending in $d$ and 
\begin{equation*}
	\mathcal{L}_-:= \langle d^k\{d(u \ \shuffle_0 \ v) -d u \ \shuffle_0 \ v - u \ \shuffle_0 \ dv\} \ |\ k \in \mathbb{N}_0,\ u,v \in L^*\ \rangle_{(\mathbb{Q}\langle L\rangle,\shuffle_0)},
\end{equation*}
that is, to be the two-sided ideal of $(\mathbb{Q}\langle L\rangle,\shuffle_0)$ algebraically generated by the above elements. We consider the $\mathbb{Q}$-linear subspace
\begin{equation*}
	\mathcal{S}_-:=\mathcal{T}_-+\mathcal{L}_-
\end{equation*}
 of $\mathbb{Q}\langle L\rangle$ generated by $\mathcal{L}_-$ and $\mathcal{T}_-$. This $\mathcal{S}_-$ also forms a two-sided ideal as our previous $\mathcal{T}+\mathcal{L}$. We put the quotient
\begin{equation*}
	\mathcal{H}_0 := \mathbb{Q}\langle L \rangle /\mathcal{S}_-.
\end{equation*}
Actually $\mathcal{H}_0$ forms a connected, filtered, commutative and cocommutative Hopf algebra (cf. \cite{EMS1} \S3.3.6), whose product is equal to $\shuffle_0$ and whose coproduct is given by 
\begin{equation*}
	\Delta_0(w) := \sum_{\substack{S \subset [n]\\ S:{\rm admissible}}}w_S \otimes w_{\overline{S}},
\end{equation*}
for $w \in Y\setminus \{{\bf 1}\} (\subset \mathcal{H}_0)$. In the summation, $S$ may be empty. we put $n:={\rm wt}(w)$, $[n] := \{1,\dots,n\}$ and $\overline{S} := [n]\setminus S$. For $w:=x_1\cdots x_n\ (x_i\in L^*,\ i=1,\dots,n)$ and $S:=\{i_1,\dots,i_k\}$ with $1\leq i_1<\cdots<i_k\leq n$, we define $w_S:=x_{i_1}\cdots x_{i_k}$. We call the set $S$ {\it admissible} if both $w_S, w_{\overline{S}} \in Y$. See \cite{EMS1} \S 3.3.8 for combinatorial method using polygons to compute $\Delta_0(w)$. We define $\mathbb{Q}$-linear map $\tilde{\Delta}_0:\mathcal{H}_0\rightarrow\mathcal{H}_0\otimes\mathcal{H}_0$ by
\begin{equation}\label{eqn:2.1.0}
	\tilde{\Delta}_0(w) := \Delta_0(w)-1\otimes w-w\otimes 1 \quad (w \in Y),
\end{equation}
and we call $\tilde{\Delta}_0$ the {\it reduced product}.
\subsection{An explicit formula for the reduced coproduct $\tilde{\Delta}_0$}
We show an explicit formula (Proposition \ref{cor:2.1.1}) to calculate the reduced coproduct $\tilde{\Delta}_0$ in this subsection. This proposition is important to prove the recurrence formula of $\zeta_{\scalebox{0.5}{\rm EMS}}(-k_1,\dots, -k_n)$ in \S3.

We consider the bilinear map $f:\mathbb{Q}\langle L \rangle \times \mathbb{Q}\langle L \rangle^{\otimes2}\rightarrow \mathbb{Q}\langle L \rangle^{\otimes2}$ defined by 
\begin{align*}
	f({\bf 1},w\otimes w')&:=w\otimes w', \\
	f(d,w\otimes w')&:=d w\otimes w'+w\otimes d w', \\
	f(y,w\otimes w')&:=y w\otimes w'+w\otimes y w',
	\intertext{and inductively}
	f(xx_0,w\otimes w')&:=f\left(x,f(x_0,w\otimes w')\right),
\end{align*}
for $w,w' \in \mathbb{Q}\langle L \rangle$, $x_0\in L$ and $x\in L^*$. Then the following lemma holds:
\begin{lem}
	There is a map $\overline{f}:\mathbb{Q}\langle L \rangle\times\mathcal{H}_0^{\otimes2}\rightarrow\mathcal{H}_0^{\otimes2}$ which makes the following diagram commutative:
	\begin{equation*}
		\begin{array}{ccc}
			\mathbb{Q} \langle L\rangle\otimes\mathbb{Q} \langle L\rangle &\stackrel{f(x,\cdot)}{\xlongrightarrow[]{\hspace{3em}}}& \mathbb{Q} \langle L\rangle\otimes\mathbb{Q} \langle L\rangle\\
			\pi\raisebox{3mm}{\rotatebox{-90}{$\twoheadlongrightarrow$}} &&\raisebox{3mm}{\rotatebox{-90}{$\twoheadlongrightarrow$}}\pi\\
			\mathcal{H}_0\otimes\mathcal{H}_0 &\stackrel{\overline{f}(x,\cdot)}{\xlongrightarrow[]{\hspace{3em}}}&\mathcal{H}_0\otimes\mathcal{H}_0
		\end{array}
	\end{equation*}
	where $x\in \mathbb{Q}\langle L\rangle$ and $\pi:\mathbb{Q}\langle L\rangle^{\otimes2}\rightarrow\mathcal{H}_0^{\otimes2}$ is a natural projection.
\end{lem}
\begin{proof}
	It is sufficient to prove $f(x,\ker{\pi})\subset \ker{\pi}$ for $x\in L^*$. Here $\ker{\pi}=\mathbb{Q}\langle L\rangle\otimes\mathcal{S}_-+\mathcal{S}_-\otimes\mathbb{Q}\langle L\rangle$. We show this by induction on ${\rm wt}(x)$. Let $x_0=d\ \mbox{or}\ y$ and put $v\in \mathcal{S}_-$. If $v\in \mathcal{T}_-$, it is clear that $x_0v\in \mathcal{T}_-\subset \mathcal{S}_-$. If $v\in \mathcal{L}_-$, for $x_0=d$ it is easy to see that $dv\in \mathcal{L}_-\subset \mathcal{S}_-$ by the definition of $\mathcal{L}_-$. Because $\mathcal{L}_-$ is a two-sided ideal of $(\mathbb{Q} \langle L\rangle,\shuffle_0)$, we have $y\shuffle_0v \in L_-$ for $x_0=y$. By the definition of $\shuffle_0$, we get 
	\begin{equation*}
		y\shuffle_0v=y(1\shuffle_0v)=y v \in \mathcal{L}_-\subset \mathcal{S}_-.
	\end{equation*}
	Because $\mathcal{S}_-$ is $\mathcal{L}_-+\mathcal{T}_-$, for $v\in \mathcal{S}_-$ and $x_0=d\ {\rm or}\ y$, we have $x_0v\in\mathcal{S}_-$.
	
	Let $\ w\in L^*$ and $v\in \mathcal{S}_-$. Then $x_0v\in \mathcal{S}_-$, so we have
	\begin{align*}
		\pi\left(f(x_0,w\otimes v)\right)&=\pi(x_0w\otimes v +w\otimes x_0v)\\
		&=\pi(x_0w\otimes v) +\pi(w\otimes x_0v)\\
		&=0.
	\end{align*}
	Let $w\in L^*$ and $v\in\mathcal{S}_-$. For $x\in L^*$, we get
	\begin{align*}
		\pi\bigl(f(xx_0,w\otimes v)\bigr)&=\pi\Bigl(f\bigl(x,f(x_0,w\otimes v)\bigr)\Bigr)\\
		&=\pi\bigl(f(x,x_0w\otimes v+w\otimes x_0v)\bigr)\\
		&=\pi\bigl(f(x,x_0w\otimes v)\bigr)+\pi\bigl(f(x,w\otimes x_0v)\bigr)\\
		&=0,
	\end{align*}
	by our induction assumption. 
	This also applies to the case when $w\in \mathcal{S}_-$ and $\ v\in L^*$, so the claim holds.
\end{proof}
For $x\in L^*$ and $w,w'\in Y$, we simply denote $\overline{f}(x,w\otimes w')$ by $x\bullet (w\otimes w')$ and we define 
\begin{equation*}
	w \otimes_{{\rm sym}} w' := w \otimes w' + w' \otimes w \in \mathcal{H}_0\otimes\mathcal{H}_0.
\end{equation*}
Then, the following equations hold in $\mathcal{H}_0\otimes\mathcal{H}_0$:
\begin{align}
	\label{eqn:2.1.1} d^n\bullet(w\otimes_{\rm sym} w') &= \sum_{i+j=n}\binom{n}{i}d^i w\otimes_{\rm sym} d^j w', \\
	\label{eqn:2.1.2} (d^n y)\bullet(w\otimes_{\rm sym} w') &= \sum_{i+j=n}\binom{n}{i}\sum_{\{u,v\}=\{d^iy,d^j\}}u w\otimes_{\rm sym}v w',
\end{align}
for $n \in \mathbb{N}$,\ $w,\ w' \in Y$. These equations can be proved inductively on $n \in \mathbb{N}$.
\begin{prop}\label{prop:2.1.1}
	For $w \in Y \setminus \{{\bf 1}\}$,
	\begin{align}
		\label{eqn:2.1.3}\tilde{\Delta}_0(d w) &= d\bullet\tilde{\Delta}_0(w), \\
		\label{eqn:2.1.4}\tilde{\Delta}_0(y w) &= y\bullet\tilde{\Delta}_0(w)+ y \otimes_{{\rm sym}} w.
	\end{align}
\end{prop}

\begin{proof}
	Let $w$ be in $Y \setminus \{{\bf 1}\}$. By the definition of $\Delta_0$ and the equation (\ref{eqn:2.1.0}), we have
	\begin{align*}
		\tilde{\Delta}_0(d w)&=\Delta_0(d w)-1\otimes_{\rm sym}d w\\
		&=\sum_{\substack{S \subset [n+1]\\ S:{\rm admissible}}}(d w)_S \otimes (d w)_{\overline{S}}-1\otimes_{\rm sym}d w\\
		&=\sum_{\substack{1\in S \subset [n+1]\\ S:{\rm admissible}}}(d w)_S \otimes (d w)_{\overline{S}}+\sum_{\substack{1\notin S \subset [n+1]\\ S:{\rm admissible}}}(d w)_S \otimes (d w)_{\overline{S}}-1\otimes_{\rm sym}d w\\
		&=\sum_{\substack{S \subset [n]\\ S:{\rm admissible}}}d\cdot w_S \otimes w_{\overline{S}}+\sum_{\substack{S \subset [n]\\ S:{\rm admissible}}}w_S \otimes d\cdot w_{\overline{S}}-(d\otimes_{\rm sym}w+1\otimes_{\rm sym}d w)\\
		&=\sum_{\substack{S \subset [n]\\ S:{\rm admissible}}}(d\cdot w_S \otimes w_{\overline{S}}+w_S \otimes d\cdot w_{\overline{S}})-(d\otimes_{\rm sym}w+1\otimes_{\rm sym}d w)\\
		&=d\bullet\left(\sum_{\substack{S \subset [n]\\ S:{\rm admissible}}}w_S \otimes w_{\overline{S}}-1\otimes_{\rm sym}w\right)\\
		&=d\bullet\tilde{\Delta}_0(w).\\
	\end{align*}
	We use $d\otimes_{\rm sym}w=0$ in $\mathcal{H}_0\otimes\mathcal{H}_0$ at the fourth equality.
	The equation (\ref{eqn:2.1.4}) can be proved in the same way.
\end{proof}

\begin{prop}\label{cor:2.1.1}
	Let $w_m := d^my$ for $m \in \mathbb{N}_0$. Then for $n \in \mathbb{N}_{\geq 2}$ and $\ k_1, \dots, k_n \in \mathbb{N}_0$, we have
	\begin{align}
		&\tilde{\Delta}_0(w_{k_1}\cdots w_{k_n}) = \sum_{i_1 + j_1=k_1} \binom{k_1}{i_1}d^{i_1}y \otimes_{{\rm sym}} d^{j_1}w_{k_2}\cdots w_{k_n} \\
		&+ \sum_{p=2}^{n-1}\sum_{\substack{i_1 + j_1=k_1\\ \scalebox{0.5}{\rotatebox{90}{$\cdots$}}\\i_p + j_p=k_p}}\prod_{a=1}^p\binom{k_a}{i_a}\sum_{\substack{\{u_q,\ v_q\}=\{d^{i_q},\ d^{j_q}y\}\\1\leq q\leq p-1}}\hspace{-1cm}(u_1 \cdots u_{p-1}d^{i_p}y \otimes_{{\rm sym}} v_1 \cdots v_{p-1}d^{j_p}w_{k_{p+1}}\cdots w_{k_{n}}). \nonumber
	\end{align}
	Here $\{u_q,v_q\}=\{d^{i_q},d^{j_q}y\}$ means $(u_q,v_q)=(d^{i_q},d^{j_q}y)\ \mbox{or}\ (d^{j_q}y,d^{i_q})$.
\end{prop}

\begin{proof}
	Because we have
	\begin{equation}\label{eqn:2.1.5}
		\tilde{\Delta}_0(d^a y w) = d^a\bullet\left(y\otimes_{\rm sym}w+y\bullet\tilde{\Delta}_0(w)\right) \quad (a \in \mathbb{N}_0)
		\end{equation}
	by Proposition \ref{prop:2.1.1}, we compute
	\begin{align}
		&\tilde{\Delta}_0(w_{k_1}w_{k_2}\cdots w_{k_n}) \nonumber \\
		&= d^{k_1}\bullet(y \otimes_{{\rm sym}} w_{k_2}\cdots w_{k_n})+ (d^{k_1}y)\bullet\tilde{\Delta}_0(w_{k_2}\cdots w_{k_n}) \nonumber \\
		&= d^{k_1}\bullet(y \otimes_{{\rm sym}} w_{k_2}\cdots w_{k_n})+ (d^{k_1}yd^{k_2})\bullet(y \otimes_{{\rm sym}} w_{k_3}\cdots w_{k_n}) \nonumber \\
		&\quad +(d^{k_1}yd^{k_2}y)\bullet\tilde{\Delta}_0(w_{k_3}\cdots w_{k_n}). \nonumber \\
		\intertext{By using the equation (\ref{eqn:2.1.5}) repeatedly, we get}
		&= \sum_{p=1}^{n-1}(d^{k_1}y\cdots yd^{k_p})\bullet(y \otimes_{{\rm sym}} w_{k_{p+1}}\cdots w_{k_n}) \nonumber \\
		&\quad +(d^{k_1}y\cdots d^{k_{n-1}}y)\bullet\tilde{\Delta}_0(w_{k_n}). \nonumber
	\end{align}
	Because $\tilde{\Delta}_0(d^a y) = 0\ (a \in \mathbb{N}_0)$ by the definition of $\tilde{\Delta}_0$, the second term vanishes. Therefore by (\ref{eqn:2.1.1}), we get
	\begin{align*}
		&\tilde{\Delta}_0(w_{k_1}w_{k_2}\cdots w_{k_n}) \\
		=&\sum_{p=1}^{n-1}\left(d^{k_1}y\cdots d^{k_{p-1}}y\right)\bullet\left(\sum_{i_p+j_p=k_p}\binom{k_p}{i_p}d^{i_p}y \otimes_{{\rm sym}} d^{j_p}w_{k_{p+1}}\cdots w_{k_n}\right). \\
		\intertext{And by using (\ref{eqn:2.1.2}) repeatedly, we have}
		=& \sum_{i_1 + j_1=k_1} \binom{k_1}{i_1}d^{i_1}y \otimes_{{\rm sym}} d^{j_1}w_{k_2}\cdots w_{k_n} \\
		&+ \sum_{p=2}^{n-1}\sum_{\substack{i_1 + j_1=k_1\\ \scalebox{0.5}{\rotatebox{90}{$\cdots$}}\\i_p + j_p=k_p}}\prod_{a=1}^p\binom{k_a}{i_a}\sum_{\substack{\{u_q,\ v_q\}=\{d^{i_q},\ d^{j_q}y\}\\1\leq q\leq p-1}}\hspace{-1cm}(u_1 \cdots u_{p-1}d^{i_p}y \otimes_{{\rm sym}} v_1 \cdots v_{p-1}d^{j_p}w_{k_{p+1}}\cdots w_{k_{n}}). 
	\end{align*}
\end{proof}
\subsection{The algebraic Birkhoff decomposition and renormalized values}
We explain the algebraic Birkhoff decomposition. This decomposition is a fundamental tool in a work of Connes and Kreimer \cite{CK} on their Hopf algebraic approach to renormalization of perturbative quantum field theory.
 This decomposition is necessary to define renormalized values. \\
\indent
Based on \cite{Man}, we recall the algebraic Birkhoff decomposition. We denote the product and the unit of $\mathbb{Q}$-algebra $\mathcal{A}$ by $m_{\mathcal{A}}$ and $u_{\mathcal{A}}$.  For a Hopf algebra $\mathcal{H}$ over $\mathbb{Q}$, we mean $\Delta_{\mathcal{H}}$, $\varepsilon_{\mathcal{H}}$ and $S_{\mathcal{H}}$ to be its coproduct, its counit and its antipode respectively. In this paper, we often use Sweedler's notation:
\begin{equation}\label{eqn:2.2.1}
	\tilde{\Delta}_0(w) := \sum_{(w)}w'\otimes w''.
\end{equation}

Let $\mathcal{H}$ be a Hopf algebra over $\mathbb{Q}$, $\mathcal{A}$ be a $\mathbb{Q}$-algebra and $\mathcal{L}(\mathcal{H},\mathcal{A})$ be the set of $\mathbb{Q}$-linear maps from $\mathcal{H}$ to $\mathcal{A}$. We define the {\it convolution} $\phi*\psi \in \mathcal{L}(\mathcal{H},\mathcal{A})$ by
	\begin{equation*}
		\phi*\psi:=m_{\mathcal{A}}\circ(\phi\otimes\psi)\circ\Delta_{\mathcal{H}}
	\end{equation*}
	for $\mathbb{Q}$-linear maps $\phi\ \mbox{and}\ \psi \in \mathcal{L}(\mathcal{H},\mathcal{A})$.
Let $\mathcal{H}$ be a Hopf algebra over $\mathbb{Q}$ and $\mathcal{A}$ be a $\mathbb{Q}$-algebra. The subset
\begin{equation*}
	G(\mathcal{H},\mathcal{A}) := \{\phi \in \mathcal{L}(\mathcal{H},\mathcal{A})\ |\ \phi({\bf 1})={\bf 1}_{\mathcal{A}}\}
\end{equation*}
endowed with the above convolution product $*$ forms a group. The unit is given by a map $e =u_{\mathcal{A}}\circ\varepsilon_{\mathcal{H}}$.\\
\indent
Let $\mathcal{H}$ be a connected filtered Hopf algebra over $\mathbb{Q}$, that is, $\mathcal{H}$ has a filtration of $\mathbb{Q}$-linear subspace:
	$$\mathcal{H}^0\subset \mathcal{H}^1\subset \cdots \subset\mathcal{H}^n \subset \bigcup_{n\in\mathbb{N}_0}\mathcal{H}^n=\mathcal{H}$$
with $\mathcal{H}^0=\mathbb{Q}$ and with the conditions:
	$\mathcal{H}^m\mathcal{H}^n\subset\mathcal{H}^{m+n}$ and 
	$S_{\mathcal{H}}(\mathcal{H}^n)\subset\mathcal{H}^n$ and 
	$\Delta_{\mathcal{H}}(\mathcal{H}^n)\subset\displaystyle\sum_{p+q=n}\mathcal{H}^p\otimes\mathcal{H}^q$
for $m,n\in\mathbb{N}_0$.

 Let $\mathcal{A}:=\mathbb{Q}[\frac{1}{z},z]]:=\mathbb{Q}[[z]][\frac{1}{z}]$ be the algebra consisting of all Laurent series. And we decompose it as $\mathcal{A}=\mathcal{A}_-\oplus\mathcal{A}_+$ where $\mathcal{A}_-:={\frac{1}{z}\mathbb{Q}[\frac{1}{z}]}$ and $\mathcal{A}_+:=\mathbb{Q}[[z]]$. We define a projection $\pi:\mathcal{A}\rightarrow\mathcal{A}_-$ by 
\begin{equation*}
	\pi\left(\sum_{n=-k}^{\infty}a_n z^n\right):=\sum_{n=-k}^{-1}a_n z^n,
\end{equation*}
with $a_n \in\mathbb{Q}$ and $k\in\mathbb{Z}$. Here we use the convention the sum over empty set is zero.

The following theorem is the fundamental tool of Connes and Kreimer (\cite{CK}) in the renormalization procedure of perturbative quantum field theory.
\begin{thm}[\cite{CK}, \cite{EMS1}, \cite{Man}: {\bf algebraic Birkhoff decomposition}]\label{thm:2.2.1}
	 For $\phi \in G(\mathcal{H},\mathcal{A})$, there are unique linear maps $\phi_+:\mathcal{H}\rightarrow \mathcal{A}_+$ and $\phi_-:\mathcal{H}\rightarrow \mathbb{Q}\oplus\mathcal{A}_-$ with $\phi_-({\bf 1})=1\in\mathbb{Q}$ such that
	\begin{equation*}
		\phi=\phi_-^{-1}*\phi_+.
	\end{equation*}
	Moreover the maps $\phi_-$ and $\phi_+$ are algebra homomorphisms if $\phi$ is an algebra homomorphism.
\end{thm}
We define the $\mathbb{Q}$-linear map $\phi:\mathcal{H}_0 \rightarrow \mathcal{A}$ by $\phi({\bf 1}):=1$ and for $k_1,\dots,k_n \in \mathbb{N}_0$,
\begin{equation}
	d^{k_1}y\cdots d^{k_n}y \mapsto \phi(d^{k_1}y\cdots d^{k_n}y)(z) := \partial^{k_1}_z\left(x\partial^{k_2}_z\right)\cdots\left(x\partial^{k_n}_z\right)\left(x(z)\right)
\end{equation}
where $x:=x(z) := \frac{e^z}{1-e^z} \in \mathcal{A}$ and $\partial_z$ is the derivative by $z$.
\begin{prop}[\cite{EMS1} \S4.2]
The $\mathbb{Q}$-linear map $\phi:\mathcal{H}_0 \rightarrow \mathcal{A}$ is well-defined and forms algebra homomorphism. Moreover, the following diagram is commutative:
$$\xymatrix{(\mathcal{H}_0,\shuffle_0) \ar[r]^{\zeta^{\shuffle}_t} \ar@{->}[rd]_{\phi} & (\mathbb{Q}[[t]],\cdot) \ar@{->}[d]^{\huge t\mapsto e^z}\\ & (\mathcal{A},\cdot) }$$
where $\zeta^{\shuffle}_t$ is the map in {\rm (\ref{eqn:2.1.6})}.
\end{prop}
Because the map $\phi$ is algebraic by the above proposition, we obtain the algebraic map:
\begin{equation}\label{eqn:2.3.1}
\phi_+:\mathcal{H}_0\rightarrow \mathcal{A}_+
\end{equation}
which is an algebra homomorphism by Theorem \ref{thm:2.2.1}.
\begin{definition}[\cite{EMS1} \S 4.2]\label{def:2.3.1}
	The {\it renormalized value} \footnote{If we follow the notations of \cite{EMS1}, it should be denoted by $\zeta_+(-k_n,\dots,-k_1)$.} $\zeta_{\scalebox{0.5}{\rm EMS}}(-k_1,\dots, -k_n)$ is defined by
	\begin{equation}
		\zeta_{\scalebox{0.5}{\rm EMS}}(-k_1,\dots, -k_n) := \lim_{z \rightarrow 0}\phi_+(d^{k_n}y\cdots d^{k_1}y)(z)
	\end{equation}
	for $k_1,\dots,k_n \in \mathbb{N}_0$.
\end{definition}
It is remarkable that the renormalized values coincide with special values of the meromorphic continuation of MZFs at non-positive arguments which do not locate at their singularities.
\begin{prop}[\cite{EMS1} Theorem 4.3]
	For $k_1 \in \mathbb{N}_0$, we have
	\begin{equation*}
		\zeta_{\scalebox{0.5}{\rm EMS}}(-k_1) = \zeta(-k_1)
	\end{equation*}
	and for $k_1,k_2 \in \mathbb{N}_0$ with $k_1+k_2$ odd, we have
	\begin{equation*}
		\zeta_{\scalebox{0.5}{\rm EMS}}(-k_1,-k_2) = \zeta(-k_1,-k_2).
	\end{equation*}
\end{prop}
We remind that, as is showed in the set (\ref{eqn:0.2}), $\zeta(s_1,\cdots,s_n)$ is always irregular at $(s_1,\cdots,s_n)=(-k_1,\cdots,-k_n) \in \mathbb{Z}_{<0}^n$ for $n\geq3$.

Another remarkable property of the renormalized values is that a certain shuffle relation hold for them. Because $\shuffle_0$ is the product of $\mathcal{H}_0$ and $\phi_+:\mathcal{H}_0\rightarrow\mathbb{Q}[[z]]$ is a unital algebra homomorphism by Theorem \ref{thm:2.2.1}, we obtain the following proposition:
\begin{prop}[\cite{EMS1} \S4.2: {\bf shuffle relation}]\label{prop:2.3.1}
	For $w,w'\in Y$, we have
	$$\phi_+(w\shuffle_0w')=\phi_+(w)\phi_+(w').$$
\end{prop}
Here are examples in lower depth:
\begin{example}
	For $a,b,c \in \mathbb{N}_0$, we have 
	\begin{align*}
		&\zeta_{\scalebox{0.5}{\rm EMS}}(-a)\cdot\zeta_{\scalebox{0.5}{\rm EMS}}(-b)=\left\{\begin{array}{lc}
			\displaystyle\sum_{k=0}^a(-1)^k\binom{a}{k}\zeta_{\scalebox{0.5}{\rm EMS}}(-b-k,-a+k)&\mbox{if $b\geq1$},\\
			\zeta_{\scalebox{0.5}{\rm EMS}}(-a,0)&\mbox{if $b=0$},\end{array}\right.\\
		&\zeta_{\scalebox{0.5}{\rm EMS}}(-a)\cdot\zeta_{\scalebox{0.5}{\rm EMS}}(-b,-c)=\left\{\begin{array}{lc}
			\displaystyle\sum_{k=0}^c(-1)^k\binom{c}{k}\zeta_{\scalebox{0.5}{\rm EMS}}(-b,-c-k,-a+k)&\mbox{if $c\geq1$},\\
			\displaystyle\sum_{k=0}^c(-1)^k\binom{c}{k}\zeta_{\scalebox{0.5}{\rm EMS}}(-b-k,-a+k,0)&\mbox{if $b\geq1$, $c=0$},\\
			\zeta_{\scalebox{0.5}{\rm EMS}}(-a,0,0)&\mbox{if $b=c=0$}.\end{array}\right.
	\end{align*}
	For our comparison, we remind below the usual shuffle relation for positive arguments. For $a,b\in\mathbb{N}_{>1}$,
	\begin{align*}
		\zeta(a)\cdot\zeta(b)=&\displaystyle\sum_{k=0}^{a-1}\binom{b-1+k}{k}\zeta(a-k,b+k)+\sum_{k=0}^{b-1}\binom{a-1+k}{k}\zeta(b-k,a+k),\\
		\intertext{and for $a,c\in\mathbb{N}_{>1}$ and $b\in\mathbb{N}$,}
		\zeta(a)\cdot\zeta(b,c)=&\displaystyle\sum_{k=0}^{a-1}\sum_{i=0}^{a-k-1}\binom{c-1+k}{k}\binom{b-1+i}{i}\zeta(a-k-i,b+i,c+k)\\
		&+\sum_{k=0}^{a-1}\sum_{j=0}^{b-1}\binom{c-1+k}{k}\binom{a-k-1+j}{j}\zeta(b-j,a-k+j,c+k)\\
		&+\sum_{k=0}^{c-1}\binom{a-1+k}{k}\zeta(b,c-k,a+k).
	\end{align*}
\end{example}
\section{Main results}\label{sec:3}
In this section, we prove a recurrence formula among renormalized values of MZFs in Proposition \ref{thm:3.1.1}. Moreover, by showing that the renormalized value $\zeta_{\scalebox{0.5}{\rm EMS}}(-k_1,\dots, -k_n)$ satisfies the recurrence formula similar to the one (\ref{eqn:1.2.2}) for $\zeta_{\scalebox{0.5}{\rm FKMT}}(-k_1,\dots, -k_n)$, we  prove an equivalence between the desingularized values and the renormalized values in Theorem \ref{thm:3.2.1}. As a corollary of Theorem \ref{thm:3.2.1}, we obtain an explicit formula of renormalized values (Corollary \ref{cor:3.2.1}).
\subsection{Recurrence formulas among renormalized values}
The goal of this subsection is to prove Proposition \ref{thm:3.1.1} which is on recurrence formula among renormalized values.

We start with the following key lemma of \cite{EMS1} which is a method to compute recursively the image of $\phi_+$ (the equation (\ref{eqn:2.3.1})).
\begin{lem}[\cite{EMS1} Corollary 4.4]\label{lem:3.1.1}
	For $w \in Y$ with ${\rm dp}(w)>1$, we have
	\begin{equation*}
		\phi_+(w) = \frac{1}{2^{{\rm dp}(w)}-2}\sum_{(w)}\phi_+(w')\phi_+(w'').
	\end{equation*}
	Here we use Sweedler's notation (\ref{eqn:2.2.1}).
\end{lem}
\begin{prop}\label{prop:3.1.1}
	For $n \in \mathbb{N}_{\geq 2}$ and $k_1, \dots, k_n \in \mathbb{N}_0$, we have
	\begin{align}
		\label{eqn:3.1.1} &\zeta_{\scalebox{0.5}{\rm EMS}}(-k_1, \dots, -k_n) = \frac{1}{2^{n-1}-1}\left\{ \sum_{i_n + j_n=k_n}\binom{k_n}{i_n}\zeta_{\scalebox{0.5}{\rm EMS}}(-i_n) \zeta_{\scalebox{0.5}{\rm EMS}}(-k_1,\dots, -k_{n-1}-j_n)\right. \\
		&+ \sum_{p=2}^{n-1}\sum_{\substack{i_n + j_n=k_n\\ \scalebox{0.5}{\rotatebox{90}{$\cdots$}}\\i_{p} + j_{p}=k_{p}}}\prod_{a=p}^n\binom{k_{a}}{i_{a}} \nonumber \\
		&\left. \times\sum_{\substack{\{\circ_q,\ \diamond_q\}=\{+,\ \scalebox{2}{,}\ \}\\p\leq q\leq n-1}}\zeta_{\scalebox{0.5}{\rm EMS}}(-i_{p}\circ_{p}\cdots \circ_{n-1} -i_n)\zeta_{\scalebox{0.5}{\rm EMS}}(-k_{1},\dots, -k_{p-1}-j_{p}\diamond_{p}  \cdots\diamond_{n-1} -j_n)\right\}. \nonumber
	\end{align}
\end{prop}

\begin{proof}
	By Proposition \ref{cor:2.1.1} and Lemma \ref{lem:3.1.1}, for $n \in \mathbb{N}_{\geq 2}$ and $k_1, \dots, k_n \in \mathbb{N}_0$ we get 
	\begin{align*}
		&\phi_+(w_{k_n}\cdots w_{k_1}) = \frac{1}{2^{n-1}-1}\left\{\vphantom{\sum_{\substack{i_1 + j_1=k_1\\ \scalebox{0.5}{\rotatebox{90}{$\cdots$}}\\i_p + j_p=k_p}}}\sum_{i_n + j_n=k_n} \binom{k_n}{i_n}\phi_+(d^{i_n}y)\phi_+(d^{j_n}w_{k_{n-1}}\cdots w_{k_1})\right. \\
		&\left.+ \sum_{p=2}^{n-1}\sum_{\substack{i_n + j_n=k_n\\ \scalebox{0.5}{\rotatebox{90}{$\cdots$}}\\i_p + j_p=k_p}}\prod_{a=p}^n\binom{k_a}{i_a}\sum_{\substack{\{u_q,\ v_q\}=\{d^{i_q},\ d^{j_q}y\}\\p+1\leq q\leq n}}\hspace{-1cm}\phi_+(u_{n} \cdots u_{p+1}d^{i_p}y)\phi_+(v_n \cdots v_{p+1}d^{j_p}w_{k_{p-1}}\cdots w_{k_{1}})\right\},
	\end{align*}
	because ${\rm dp}(w) = n$. For $p\leq q\leq n-1$, we define 
	\begin{equation*}
		(\circ_q,\diamond_q) :=\left\{\begin{array}{ll}
			(+,\scalebox{2}{,}) & \mbox{if}\ (u_{q+1},v_{q+1}) = (d^{i_{q+1}},d^{j_{q+1}}y),\\
			(\scalebox{2}{,},+) & \mbox{if}\ (u_{q+1},v_{q+1}) = (d^{j_{q+1}}y,d^{i_{q+1}}).
		\end{array}\right.
	\end{equation*}
	Then by the definition of $\zeta_{\scalebox{0.5}{\rm EMS}}(-k_1, \dots, -k_n)$, the equation (\ref{eqn:3.1.1}) holds.
\end{proof}
We define the following generating functions in $\mathbb{C}[[x]]$ for $n \in \mathbb{N}_{\geq 2}$ and $k_1, \dots, k_n \in \mathbb{N}_0$:
\begin{align*}
	\frak{h}:=\frak{h}(x) &:= \sum_{k_1=0}^{\infty}\frac{(-x)^{k_1}}{k_1!}\zeta_{\scalebox{0.5}{\rm EMS}}(-k_1), \\
	\frak{h}_{k_1,\dots,k_{n-1}}(x) &:= \sum_{k_n=0}^{\infty}\frac{(-x)^{k_n}}{k_n!}\zeta_{\scalebox{0.5}{\rm EMS}}(-k_1,\dots,-k_n), \\
	\overline{\frak{h}}_{k_1,\dots,k_n}(x) &:= \partial_{x}^{k_n}\frak{h}_{k_1,\dots,k_{n-1}}(x).
\end{align*}
Here for $n\in\mathbb{N}$, we set $\frak{h}_{k_1,\dots,k_{n-1}}(x):=\frak{h}(x)$.

The equation (\ref{eqn:3.1.1}) looks complicated. But it can be simplified to the following recurrence formula (\ref{eqn:3.1.2}).
\begin{prop}\label{thm:3.1.1}
	For $n \in \mathbb{N}_{\geq 2}$ and $k_1, \dots, k_n \in \mathbb{N}_0$, we have
	\begin{equation}\label{eqn:3.1.2}
		\zeta_{\scalebox{0.5}{\rm EMS}}(-k_1, \dots, -k_n) = \sum_{i_n + j_n=k_n}\binom{k_n}{i_n}\zeta_{\scalebox{0.5}{\rm EMS}}(-i_n) \zeta_{\scalebox{0.5}{\rm EMS}}(-k_1,\dots, -k_{n-1}-j_n),
	\end{equation}
	and 
	\begin{equation}\label{eqn:3.1.3}
		\frak{h}_{k_1,\dots,k_{n-1}}(x)=(-1)^{k_1+\cdots+k_{n-1}}\left(\frak{h}\partial_x^{k_{n-1}}\right)\cdots\left(\frak{h}\partial_x^{k_1}\right)\left(\frak{h}\right).
	\end{equation}
\end{prop}
\begin{proof}
	We prove (\ref{eqn:3.1.2}) and (\ref{eqn:3.1.3}) by induction on $n\in \mathbb{N}_{\geq2}$. Let $n=2$. Then by the equation (\ref{eqn:3.1.1}) of Proposition \ref{prop:3.1.1}, the equation (\ref{eqn:3.1.2}) clearly holds. And by the equation (\ref{eqn:3.1.3}) for $n=2$, we have
	\begin{align}
		\label{eqn:3.1.6}\frak{h}_{k_1}(x)&=\sum_{k_2=0}^{\infty}\frac{(-x)^{k_2}}{k_2!}\zeta_{\scalebox{0.5}{\rm EMS}}(-k_1, -k_2) \\
		&=\sum_{k_2=0}^{\infty}\frac{(-x)^{k_2}}{k_2!}\sum_{i_2 + j_2=k_2}\binom{k_2}{i_2}\zeta_{\scalebox{0.5}{\rm EMS}}(-i_2) \zeta_{\scalebox{0.5}{\rm EMS}}(-k_1-j_2) \nonumber\\
		&=\left\{\sum_{i_2=0}^{\infty}\frac{(-x)^{i_2}}{i_2!}\zeta_{\scalebox{0.5}{\rm EMS}}(-i_2)\right\} \left\{\sum_{j_2=0}^{\infty}\frac{(-x)^{j_2}}{j_2!}\zeta_{\scalebox{0.5}{\rm EMS}}(-k_1-j_2)\right\} \nonumber\\
		&=\frak{h}\left\{(-1)^{k_1}\partial_x^{k_1}\left(\frak{h}\right)\right\} \nonumber\\
		&=(-1)^{k_1}\left(\frak{h}\partial_x^{k_1}\right)\left(\frak{h}\right).\nonumber
	\end{align}
	
	Let $n=n_0\geq3$. We assume that (\ref{eqn:3.1.2}) and (\ref{eqn:3.1.3}) hold for $2\leq n\leq n_0-1$. Firstly, we prove the equation (\ref{eqn:3.1.2}). By Lemma \ref{lem:3.1.4} which will be proved later, the second term of the right hand side of the equation (\ref{eqn:3.1.1}) is calculated to be
	\begin{align*}
		&\sum_{p=2}^{n_0-1}\sum_{\substack{\{\circ_q,\ \diamond_q\}=\{+,\ \scalebox{2}{,}\ \}\\p\leq q\leq n_0-1}}\left\{\sum_{i_{n_0} + j_{n_0}=k_{n_0}}\binom{k_{n_0}}{i_{n_0}}\zeta_{\scalebox{0.5}{\rm EMS}}(-i_{n_0}) \zeta_{\scalebox{0.5}{\rm EMS}}(-k_1,\dots, -k_{n_0-1}-j_{n_0})\right\} \\
		&=\sum_{p=2}^{{n_0}-1}2^{{n_0}-p}\left\{\sum_{i_{n_0} + j_{n_0}=k_{n_0}}\binom{k_{n_0}}{i_{n_0}}\zeta_{\scalebox{0.5}{\rm EMS}}(-i_{n_0}) \zeta_{\scalebox{0.5}{\rm EMS}}(-k_1,\dots, -k_{{n_0}-1}-j_{n_0})\right\} \\
		&=(2^{{n_0}-1}-2)\sum_{i_{n_0} + j_{n_0}=k_{n_0}}\binom{k_{n_0}}{i_{n_0}}\zeta_{\scalebox{0.5}{\rm EMS}}(-i_{n_0}) \zeta_{\scalebox{0.5}{\rm EMS}}(-k_1,\dots, -k_{{n_0}-1}-j_{n_0}).
	\end{align*}
	Therefore, we have
	\begin{align*}
		&(\mbox{RHS of (\ref{eqn:3.1.1})}) \\
		=& \frac{1}{2^{{n_0}-1}-1}\left\{\sum_{i_{n_0} + j_{n_0}=k_{n_0}}\binom{k_{n_0}}{i_{n_0}}\zeta_{\scalebox{0.5}{\rm EMS}}(-i_{n_0}) \zeta_{\scalebox{0.5}{\rm EMS}}(-k_1,\dots, -k_{{n_0}-1}-j_{n_0})\right. \\
		&\left.+ (2^{{n_0}-1}-2)\sum_{i_{n_0} + j_{n_0}=k_{n_0}}\binom{k_{n_0}}{i_{n_0}}\zeta_{\scalebox{0.5}{\rm EMS}}(-i_{n_0}) \zeta_{\scalebox{0.5}{\rm EMS}}(-k_1,\dots, -k_{{n_0}-1}-j_{n_0})\right\} \\
		&=\sum_{i_{n_0} + j_{n_0}=k_{n_0}}\binom{k_{n_0}}{i_{n_0}}\zeta_{\scalebox{0.5}{\rm EMS}}(-i_{n_0}) \zeta_{\scalebox{0.5}{\rm EMS}}(-k_1,\dots, -k_{{n_0}-1}-j_{n_0}).
	\end{align*}
	So we get the equation (\ref{eqn:3.1.2}) for $n\geq3$.
	
	Secondly, we prove the equation (\ref{eqn:3.1.3}) for $n=n_0\geq3$. By using the equation (\ref{eqn:3.1.2}) for $n=n_0$ which we have proved just above, we have
	\begin{align*}
		\frak{h}_{k_1,\cdots,k_{{n_0}-1}}(x)&=(-1)^{k_{{n_0}-1}}\frak{h}(x)\partial_x^{k_{{n_0}-1}}\left(\frak{h}_{k_1,\cdots,k_{{n_0}-2}}(x)\right)
		\intertext{in the same way to case of $n=2$. By our induction hypotheses,}
		&=(-1)^{k_{{n_0}-1}}\frak{h}(x)\partial_x^{k_{{n_0}-1}}\Bigl((-1)^{k_1+\cdots+k_{{n_0}-2}}\left(\frak{h}(x)\partial_x^{k_{{n_0}-2}}\right)\cdots\left(\frak{h}(x)\partial_x^{k_1}\right)\left(\frak{h}(x)\right)\Bigr) \\
		&=(-1)^{k_1+\cdots+k_{{n_0}-1}}\left(\frak{h}(x)\partial_x^{k_{{n_0}-1}}\right)\cdots\left(\frak{h}(x)\partial_x^{k_1}\right)\left(\frak{h}(x)\right)
	\end{align*}
	So we get the equation (\ref{eqn:3.1.3}) for $n\geq3$.
\end{proof}

We prove the following lemma used in the above proof.
\begin{lem}\label{lem:3.1.4}
	Let $n_0\geq3$. We assume that {\rm (\ref{eqn:3.1.3})} holds for $n=l$ with $2\leq l\leq n_0-1$. Let $2\leq p \leq n_0-1$ and $\circ_i \in \{+,\scalebox{2}{,}\}$ for $p\leq i\leq n_0-1$. Then we have
	\begin{align}
		\label{eqn:3.1.4} \sum_{\substack{i_p + j_p=k_p\\ \scalebox{0.5}{\rotatebox{90}{$\cdots$}}\\i_{n_0} + j_{n_0}=k_{n_0}}}\prod_{a=p}^{n_0}&\binom{k_a}{i_a}\zeta_{\scalebox{0.5}{\rm EMS}}(-i_{p}\circ_{p}\cdots \circ_{{n_0}-1} -i_{n_0})\zeta_{\scalebox{0.5}{\rm EMS}}(-k_{1},\dots, -k_{p-1}-j_{p}\diamond_{p}  \cdots\diamond_{{n_0}-1} -j_{n_0}) \\
		&= \sum_{i_{n_0} + j_{n_0}=k_{n_0}}\binom{k_{n_0}}{i_{n_0}}\zeta_{\scalebox{0.5}{\rm EMS}}(-i_{n_0}) \zeta_{\scalebox{0.5}{\rm EMS}}(-k_1,\dots, -k_{{n_0}-1}-j_{n_0}). \nonumber
	\end{align}
	Here $\diamond_i$ is chosen to be with $\{\circ_i,\ \diamond_i\}=\{+,\scalebox{2}{,}\}$ for $p\leq i\leq {n_0}-1$.
\end{lem}
\noindent
\begin{proof}
	We get
	\begin{align}
		\sum_{k_{n_0}=0}^{\infty}\frac{(-x)^{k_{n_0}}}{k_{n_0}!}(\mbox{RHS of (\ref{eqn:3.1.4})})&=(-1)^{k_{{n_0}-1}}\frak{h}\partial_{x}^{k_{{n_0}-1}}\Bigl(\frak{h}_{k_1,\dots,k_{{n_0}-2}}(x)\Bigr) \nonumber
		\intertext{in the same way to the computations of $\frak{h}_{k_1}(x)$ in (\ref{eqn:3.1.6}). By our induction hypothesis on (\ref{eqn:3.1.3}), for $n_0$ we obtain}
		&=\label{eqn:3.1.5} (-1)^{k_1+\cdots+k_{{n_0}-1}}\left(\frak{h}\partial_x^{k_{{n_0}-1}}\right)\cdots\left(\frak{h}\partial_x^{k_1}\right)(\frak{h}).
	\end{align}
	On the other hand, we have
	\begin{align*}
		\sum_{k_{n_0}=0}^{\infty}&\frac{(-x)^{k_{n_0}}}{k_{n_0}!}(\mbox{LHS of (\ref{eqn:3.1.4})}) \\
		= &\sum_{\substack{i_p+ j_p=k_p\\ \scalebox{0.5}{\rotatebox{90}{$\cdots$}}\\i_{{n_0}-1} + j_{{n_0}-1}=k_{{n_0}-1}}}\prod_{a=p}^{{n_0}-1}\binom{k_a}{i_a}\left\{\sum_{i_{n_0}=0}^{\infty}\frac{(-x)^{i_{n_0}}}{i_{n_0}!}\zeta_{\scalebox{0.5}{\rm EMS}}(-i_{p}\circ_{p}\cdots \circ_{{n_0}-1} -i_{n_0})\right\} \\
		&\times \left\{\sum_{j_{n_0}=0}^{\infty}\frac{(-x)^{j_{n_0}}}{j_{n_0}!}\zeta_{\scalebox{0.5}{\rm EMS}}(-k_{1},\dots, -k_{p-1}-j_{p}\diamond_{p}  \cdots\diamond_{{n_0}-1} -j_{n_0})\right\}.
	\end{align*}
	We also consider the following two cases:
	\begin{enumerate}
	\item[{\it Case i)}] : When $(\circ_{{n_0}-1},\diamond_{{n_0}-1}) = (\scalebox{2}{,},+)$, we compute
	\begin{align*}
		&\sum_{k_{n_0}=0}^{\infty}\frac{(-x)^{k_{n_0}}}{k_{n_0}!}(\mbox{LHS of (\ref{eqn:3.1.4})}) \\
		= &\sum_{\substack{i_p+ j_p=k_p\\ \scalebox{0.5}{\rotatebox{90}{$\cdots$}}\\i_{{n_0}-1} + j_{{n_0}-1}=k_{{n_0}-1}}}\prod_{a=p}^{{n_0}-1}\binom{k_a}{i_a}\left\{\sum_{i_{n_0}=0}^{\infty}\frac{(-x)^{i_{n_0}}}{i_{n_0}!}\zeta_{\scalebox{0.5}{\rm EMS}}(-i_{p}\circ_{p}\cdots \circ_{{n_0}-2}-i_{{n_0}-1}, -i_{n_0})\right\} \\
		&\times \left\{\sum_{j_{n_0}=0}^{\infty}\frac{(-x)^{j_{n_0}}}{j_{n_0}!}\zeta_{\scalebox{0.5}{\rm EMS}}(-k_{1},\dots, -k_{p-1}-j_{p}\diamond_{p}  \cdots\diamond_{{n_0}-2}-j_{{n_0}-1} -j_{n_0})\right\}.
		\intertext{Put $m := \left\{\begin{array}{ll}
		p-1 & {\rm when}\ \mbox{{\rm $\diamond_i$ is + for all $i$}}, \\
		\max{\{l\ |\ p \leq l \leq {n_0}-2,\ \diamond_l=\scalebox{2}{,}\}} & \mbox{{\rm otherwise}}.
	\end{array}\right.$}
		\intertext{Then we have}
		= &\sum_{\substack{i_p+ j_p=k_p\\ \scalebox{0.5}{\rotatebox{90}{$\cdots$}}\\i_{{n_0}-1} + j_{{n_0}-1}=k_{{n_0}-1}}}\prod_{a=p}^{{n_0}-1}\binom{k_a}{i_a}\left\{\sum_{i_{n_0}=0}^{\infty}\frac{(-x)^{i_{n_0}}}{i_{n_0}!}\zeta_{\scalebox{0.5}{\rm EMS}}(-i_{p}\circ_{p}\cdots \circ_{{n_0}-2}-i_{{n_0}-1}, -i_{n_0})\right\} \\
		&\times (-1)^{S}\partial_x^{S}\left\{\sum_{j_{n_0}=0}^{\infty}\frac{(-x)^{j_{n_0}}}{j_{n_0}!}\zeta_{\scalebox{0.5}{\rm EMS}}(-k_{1},\dots, -k_{p-1}-j_{p}\diamond_{p}  \cdots\diamond_{m-1}-j_m)\right\}.
		\intertext{Here $S:= \left\{\begin{array}{ll}
		k_{p-1}+j_p+\cdots+j_{{n_0}-1} & {\rm when}\ \mbox{{\rm $\diamond_i$ is + for all $i$}}, \\
		j_{m+1}+\cdots+j_{{n_0}-1} & \mbox{{\rm otherwise}}.\end{array}\right.$}
		=& \sum_{\substack{i_p+ j_p=k_p\\ \scalebox{0.5}{\rotatebox{90}{$\cdots$}}\\i_{{n_0}-1} + j_{{n_0}-1}=k_{{n_0}-1}}}\hspace{-2.em}(-1)^{S}\prod_{a=p}^{{n_0}-1}\binom{k_a}{i_a} \frak{h}_{i_p\circ_p\cdots \circ_{{n_0}-2} \ i_{{n_0}-1}}(x) \cdot\overline{\frak{h}}_{k_{1},\dots, k_{p-1}+j_{p}\diamond_{p}  \cdots\diamond_{{n_0}-2}\ j_{{n_0}-1}}(x).
		\intertext{Here we use the definitions of $\frak{h}_{k_1,\dots,k_{n-1}}(x)$ and $\overline{\frak{h}}_{k_1,\dots,k_{n_0}}(x)$. And by using our induction hypothesis on (\ref{eqn:3.1.3}), we have}
		=& \sum_{\substack{i_p+ j_p=k_p\\ \scalebox{0.5}{\rotatebox{90}{$\cdots$}}\\i_{{n_0}-1} + j_{{n_0}-1}=k_{{n_0}-1}}}\hspace{-1.em}\prod_{a=p}^{{n_0}-1}\binom{k_a}{i_a} \left\{(-1)^{\scalebox{0.65}{$\displaystyle\sum_{q=p}^{{n_0}-1}i_q$}}\left(\frak{h}\partial_x^{i_{{n_0}-1}}\right)\left(\frak{h}^{\delta_{{n_0}-2}}\partial_x^{i_{{n_0}-2}}\right)\cdots\left(\frak{h}^{\delta_p}\partial_x^{i_p}\right)(\frak{h})\right\} \\
		&\times \left\{(-1)^{\scalebox{0.65}{$\displaystyle\sum_{q=1}^{p-1}$}k_q + \scalebox{0.65}{$\displaystyle\sum_{q=p}^{{n_0}-1}j_q$}}\partial_{x}^{j_{{n_0}-1}}\left(\frak{h}^{1-\delta_{{n_0}-2}}\partial_x^{j_{{n_0}-2}}\right)\cdots\left(\frak{h}^{1-\delta_p}\partial_x^{j_p}\right)\left(\frak{h}\partial_x^{k_{p-1}}\right)\cdots\left(\frak{h}\partial_x^{k_1}\right)(\frak{h})\right\}.
		\intertext{Here we put $\delta_{i} := \left\{\begin{array}{cl}0&{\rm if}\ \circ_i = +, \\ 1&{\rm if}\ \circ_i = \scalebox{2}{,},\end{array}\right.$ for $p\leq i \leq {n_0}-2$.}
		=& (-1)^{\scalebox{0.65}{$\displaystyle\sum_{q=1}^{{n_0}-1}k_q$}}\frak{h}\sum_{\substack{i_p+ j_p=k_p\\ \scalebox{0.5}{\rotatebox{90}{$\cdots$}}\\i_{{n_0}-1} + j_{{n_0}-1}=k_{{n_0}-1}}}\prod_{a=p}^{{n_0}-1}\binom{k_a}{i_a} \left\{\partial_x^{i_{{n_0}-1}}\left(\frak{h}^{\delta_{{n_0}-2}}\partial_x^{i_{{n_0}-2}}\right)\cdots\left(\frak{h}^{\delta_p}\partial_x^{i_p}\right)(\frak{h})\right\} \\
		&\times \left\{\partial_{x}^{j_{{n_0}-1}}\left(\frak{h}^{1-\delta_{{n_0}-2}}\partial_x^{j_{{n_0}-2}}\right)\cdots\left(\frak{h}^{1-\delta_p}\partial_x^{j_p}\right)\left(\frak{h}\partial_x^{k_{p-1}}\right)\cdots\left(\frak{h}\partial_x^{k_1}\right)(\frak{h})\right\} \\
		=& (-1)^{\scalebox{0.65}{$\displaystyle\sum_{q=1}^{{n_0}-1}k_q$}}\frak{h}\partial_{x}^{k_{{n_0}-1}}\left(\frak{h}\hspace{-0.5em}\sum_{\substack{i_p+ j_p=k_p\\ \scalebox{0.5}{\rotatebox{90}{$\cdots$}}\\i_{{n_0}-2} + j_{{n_0}-2}=k_{{n_0}-2}}}\hspace{-1em}\prod_{a=p}^{{n_0}-2}\binom{k_a}{i_a} \left\{\partial_x^{i_{{n_0}-2}}\left(\frak{h}^{\delta_{{n_0}-3}}\partial_x^{i_{{n_0}-3}}\right)\cdots\left(\frak{h}^{\delta_p}\partial_x^{i_p}\right)(\frak{h})\right\}\right. \\
		&\times \left.\left\{\partial_{x}^{j_{{n_0}-2}}\left(\frak{h}^{1-\delta_{{n_0}-3}}\partial_x^{j_{{n_0}-3}}\right)\cdots\left(\frak{h}^{1-\delta_p}\partial_x^{j_p}\right)\left(\frak{h}\partial_x^{k_{p-1}}\right)\cdots\left(\frak{h}\partial_x^{k_1}\right)(\frak{h})\right\}\vphantom{\sum_{\substack{i_p+ j_p=k_p\\ \scalebox{0.5}{\rotatebox{90}{$\cdots$}}\\i_{n-2} + j_{n-2}=k_{n-2}}}}\right).
		\intertext{We use Leibniz rule in last equality. By using this rule repeatedly, we get}
		=& (-1)^{\scalebox{0.65}{$\displaystyle\sum_{q=1}^{{n_0}-1}k_q$}}\left(\frak{h}\partial_x^{k_{{n_0}-1}}\right)\cdots\left(\frak{h}\partial_x^{k_1}\right)(\frak{h}).
	\end{align*}
This is equal to (\ref{eqn:3.1.5}). 
	\end{enumerate}
	\begin{enumerate}
	\item[{\it Case ii)}] : When $(\circ_{{n_0}-1},\diamond_{{n_0}-1}) = (+,\scalebox{2}{,})$, it can be proved in the same way to {\it Case i)}.
	\end{enumerate}
\end{proof}
\subsection{An equivalence between desingularized values and renormalized ones}
We reveal a close relationship among desingularized values and renormalized ones in Theorem \ref{thm:3.2.1}. As a consequence, we get an explicit formula of renormalized values in terms of Bernoulli numbers in Corollary \ref{cor:3.2.1}.

Our main theorem of this paper is the following explicit relationship between the generating function $Z_{\scalebox{0.5}{\rm FKMT}}(t_1,\dots,t_n)$ of the desingularized values $\zeta_{\scalebox{0.5}{\rm FKMT}}(-k_1,\dots,-k_n)$ in (\ref{eqn:0.4}) and the generating function $Z_{\scalebox{0.5}{\rm EMS}}(t_1,\dots,t_n)$ of the renormalized values $\zeta_{\scalebox{0.5}{\rm EMS}}(-k_1,\dots,-k_n)$ in (\ref{eqn:0.5}).
\begin{thm}\label{thm:3.2.1}
	For $n \in \mathbb{N}$, we have
	\begin{equation}\label{eqn:3.2.3}
		Z_{\scalebox{0.5}{\rm EMS}}(t_1,\dots,t_n) = \prod_{i=1}^{n}\frac{1-e^{-t_i-\cdots-t_n}}{t_i+\cdots+t_n}\cdot Z_{\scalebox{0.5}{\rm FKMT}}(-t_1,\dots,-t_n).
	\end{equation}
\end{thm}
\begin{proof}
	By Proposition \ref{thm:3.1.1} and Lemma \ref{lem:3.1.4} we get
	\begin{equation}\label{eqn:3.2.1}
		\zeta_{\scalebox{0.5}{\rm EMS}}(-k_1,\dots,-k_n) = \sum_{\substack{i_2 + j_2=k_2\\ \scalebox{0.5}{\rotatebox{90}{$\cdots$}}\\i_n + j_n=k_n}}\prod_{a=2}^n\binom{k_a}{i_a}\zeta_{\scalebox{0.5}{\rm EMS}}(-i_{2},\dots, -i_n)\zeta_{\scalebox{0.5}{\rm EMS}}(-k_{1}-j_{2}- \cdots-j_n).
	\end{equation}
	Here, we use Lemma \ref{lem:3.1.4} for $p=2$ and for all $\circ_q=\scalebox{2}{,}\ (2 \leq q \leq n)$. It is remarkable that the same recurrence formula holds for $\zeta_{\scalebox{0.5}{\rm FKMT}}(-k_1,\dots,-k_n)$ of (\ref{eqn:1.2.2}). Thus, we get
	\begin{equation}\label{eqn:3.2.4}
		Z_{\scalebox{0.5}{\rm EMS}}(t_1,\dots,t_n) = Z_{\scalebox{0.5}{\rm EMS}}(t_2,\dots,t_n)\cdot Z_{\scalebox{0.5}{\rm EMS}}(t_1+\cdots+t_n) \quad(n \in \mathbb{N}).
	\end{equation}
	Now from \cite{EMS1} Theorem 4.3, $\zeta_{\scalebox{0.5}{\rm EMS}}(-k_1)=\zeta(-k_1)$ at $k_1 \in \mathbb{N}_0$, so we can write $Z_{\scalebox{0.5}{\rm EMS}}(x)$ by
	\begin{equation*}
		Z_{\scalebox{0.5}{\rm EMS}}(x) = \frac{1+x-e^{x}}{x(e^{x}-1)}.
	\end{equation*}
	We get the following equation by $Z_{\scalebox{0.5}{\rm EMS}}(x)$ and $Z_{\scalebox{0.5}{\rm FKMT}}(x)$:
	\begin{equation}\label{eqn:3.2.5}
		Z_{\scalebox{0.5}{\rm EMS}}(x) = \frac{1-e^{-x}}{x} Z_{\scalebox{0.5}{\rm FKMT}}(-x).
	\end{equation}
	By using (\ref{eqn:1.2.1}), (\ref{eqn:3.2.4}) and (\ref{eqn:3.2.5}), we get (\ref{eqn:3.2.3}).
\end{proof}
By Theorem \ref{thm:3.2.1}, we find that desingularized values and renormalized ones are equivalent. Namely, the renormalized values can be given as linear combinations of the desingularized ones.
\begin{example}
	The desingularized values and the renormalized values are equal at the origin:
	$$\zeta_{\scalebox{0.5}{\rm FKMT}}(\underbrace{0,\dots,0}_{n})=\zeta_{\scalebox{0.5}{\rm EMS}}(\underbrace{0,\dots,0}_{n})=B_1^n=\left(-\frac{1}{2}\right)^n$$
\end{example}

\begin{example}\label{ex:3.1}
	For $k_1,k_2,k_3 \in \mathbb{N}_0$, we have
	\begin{align*}
		&\zeta_{\scalebox{0.5}{\rm EMS}}(-k_1) = \displaystyle\sum_{\nu_{01}+\nu_{11}=k_1}\binom{k_1}{\nu_{01}}\frac{(-1)^{\nu_{11}}}{\nu_{01}+1}\zeta_{\scalebox{0.5}{\rm FKMT}}(-\nu_{11}),\\
		&\zeta_{\scalebox{0.5}{\rm EMS}}(-k_1,-k_2) = \displaystyle\sum_{\substack{\nu_{01}+\nu_{11}=k_1 \\\nu_{02}+\nu_{12}+\nu_{22}=k_2}}\binom{k_1}{\nu_{01}}\binom{k_2}{\nu_{02}\ \nu_{12}}\frac{1}{\nu_{02}+1}\frac{(-1)^{\nu_{11}+\nu_{22}}}{\nu_{01}+\nu_{12}+1}\zeta_{\scalebox{0.5}{\rm FKMT}}(-\nu_{11},-\nu_{22}), \\
		&\zeta_{\scalebox{0.5}{\rm EMS}}(-k_1,-k_2,-k_3) = \displaystyle\sum_{\substack{\nu_{01}+\nu_{11}=k_1 \\\nu_{02}+\nu_{12}+\nu_{22}=k_2\\\nu_{03}+\nu_{13}+\nu_{23}+\nu_{33}=k_3}}\binom{k_1}{\nu_{01}}\binom{k_2}{\nu_{02}\ \nu_{12}}\binom{k_3}{\nu_{03}\ \nu_{13}\ \nu_{23}}\\
		&\hspace{3cm}\times \frac{1}{\nu_{03}+1}\frac{1}{\nu_{02}+\nu_{13}+1}\frac{(-1)^{\nu_{01}+\nu_{12}+\nu_{23}}}{\nu_{01}+\nu_{12}+\nu_{23}+1}\zeta_{\scalebox{0.5}{\rm FKMT}}(-\nu_{11},-\nu_{22},-\nu_{33}).\\
		\end{align*}
		Here $\binom{k_2}{\nu_{02}\ \nu_{12}}:=\frac{k_2!}{\nu_{02}!\nu_{12}!(k_2-\nu_{02}-\nu_{12})!}$ and $\binom{k_3}{\nu_{03}\ \nu_{13}\ \nu_{23}}:=\frac{k_3!}{\nu_{03}!\nu_{13}!\nu_{23}!(k_3-\nu_{03}-\nu_{13}-\nu_{23})!}$.
\end{example}
On the other hand, desingularized values can be also given as linear combinations of product of  renormalized ones and Bernoulli numbers $B_n$:
\begin{example}\label{ex:3.2}
	For $k_1,k_2,k_3 \in \mathbb{N}_0$, we have
	\begin{align*}
		&\zeta_{\scalebox{0.5}{\rm FKMT}}(-k_1) = (-1)^{k_1}\displaystyle\sum_{\nu_{01}+\nu_{11}=k_1}\binom{k_1}{\nu_{01}}B_{\nu_{01}}\zeta_{\scalebox{0.5}{\rm EMS}}(-\nu_{11}), \\
		&\zeta_{\scalebox{0.5}{\rm FKMT}}(-k_1,-k_2) = (-1)^{k_1+k_2}\displaystyle\sum_{\substack{\nu_{01}+\nu_{11}=k_1 \\\nu_{02}+\nu_{12}+\nu_{22}=k_2}}\binom{k_1}{\nu_{01}}\binom{k_2}{\nu_{02}\ \nu_{12}}B_{\nu_{02}}B_{\nu_{01}+\nu_{12}}\zeta_{\scalebox{0.5}{\rm EMS}}(-\nu_{11},-\nu_{22}),\\
		&\zeta_{\scalebox{0.5}{\rm FKMT}}(-k_1,-k_2,-k_3) = (-1)^{k_1+k_2+k_3}\displaystyle\sum_{\substack{\nu_{01}+\nu_{11}=k_1 \\\nu_{02}+\nu_{12}+\nu_{22}=k_2\\\nu_{03}+\nu_{13}+\nu_{23}+\nu_{33}=k_3}}\binom{k_1}{\nu_{01}}\binom{k_2}{\nu_{02}\ \nu_{12}}\binom{k_3}{\nu_{03}\ \nu_{13}\ \nu_{23}}\\
		&\hspace{5cm}\times B_{\nu_{03}}B_{\nu_{02}+\nu_{13}}B_{\nu_{01}+\nu_{12}+\nu_{23}}\zeta_{\scalebox{0.5}{\rm EMS}}(-\nu_{11},-\nu_{22},-\nu_{33}). 
	\end{align*}
	\end{example}
By combining Proposition \ref{prop:1.1.1} and Theorem \ref{thm:3.2.1}, we obtain the following corollary.
\begin{cor}\label{cor:3.2.1}
	For $n\in\mathbb{N}$, we have
	\begin{equation*}
		Z_{\scalebox{0.5}{\rm EMS}}(t_1,\dots,t_n) = \prod_{i=1}^n\frac{(t_i+\cdots+t_n)-(e^{t_i+\cdots+t_n}-1)}{(t_i+\cdots+t_n)(e^{t_i+\cdots+t_n}-1)}.
	\end{equation*}
\end{cor}
The above equation is equivalent to the equation {\rm (\ref{eqn:0.3})}. 
Therefore the renormalized values are described explicitly in terms of  Bernoulli numbers:
\begin{example}
	For $k_1,k_2,k_3 \in \mathbb{N}_0$, we have
	\begin{align*}
	&\zeta_{\scalebox{0.5}{\rm EMS}}(-k_1) = \displaystyle\frac{(-1)^{k_1}}{k_1+1}B_{k_1+1},\\
	&\zeta_{\scalebox{0.5}{\rm EMS}}(-k_1,-k_2) = (-1)^{k_1+k_2}\displaystyle\sum_{\nu_{12}+\nu_{22}=k_2}\binom{k_2}{\nu_{12}}\frac{B_{\nu_{22}+1}}{\nu_{22}+1}\frac{B_{k_1+\nu_{12}+1}}{k_1+\nu_{12}+1},\\
	&\zeta_{\scalebox{0.5}{\rm EMS}}(-k_1,-k_2,-k_3) = (-1)^{k_1+k_2+k_3}\displaystyle\sum_{\substack{\nu_{12}+\nu_{22}=k_2\\ \nu_{13}+\nu_{23}+\nu_{33}=k_3}}\binom{k_2}{\nu_{12}}\binom{k_3}{\nu_{13}\ \nu_{23}} \\
	&\hspace{6cm}\times \frac{B_{\nu_{33}+1}}{\nu_{33}+1}\frac{B_{\nu_{22}+\nu_{23}+1}}{\nu_{22}+\nu_{23}+1}\frac{B_{k_1+\nu_{12}+\nu_{13}+1}}{k_1+\nu_{12}+\nu_{13}+1}.
	\end{align*}
\end{example}

As is explained in our introduction, other types of renormalized values were investigated in several places in the literature (\cite{EMS2}, \cite{GZ}, \cite{MP} etc). However, their explicit relationships with the desingularized values $\zeta_{\scalebox{0.5}{\rm FKMT}}(-k_1,\cdots,-k_n)$ do not seem to be shown so far, actually which was posed as a question in \cite{FKMT} Question 4.8. It would be great if our equivalence (Theorem \ref{thm:3.2.1}) could also lead a direction to settle their question.


\bigskip
\thanks{ {\it Acknowledgements}. The author is cordially grateful to Professor H. Furusho for guiding him towards this topic and for giving  useful suggestions to him. I greatly appreciate the referee's numerous and helpful comments.}

\end{document}